\documentclass{amsart}
\usepackage{amssymb,amsmath,mathrsfs,amsfonts}
\usepackage{amscd, amssymb, amsmath, amsthm, graphics,graphicx}
\usepackage{amsmath,amsfonts,amsthm,amssymb}
\usepackage{mathabx}
\usepackage{slashbox}
\usepackage{fancyhdr}
\usepackage[all]{xy}
\usepackage[colorlinks=true]{hyperref}
\usepackage[small,nohug,heads=vee]{diagrams}
\usepackage{multirow}
\usepackage{enumitem}
\diagramstyle[labelstyle=\scriptstyle]
\numberwithin{equation}{section}

\newtheorem{theorem}{Theorem}[section]
\newtheorem{lemma}[theorem]{Lemma}
\newtheorem{proposition}[theorem]{Proposition}

\theoremstyle{definition}
\newtheorem{definition}[theorem]{Definition}

\newtheorem{remark}[theorem]{Remark}

\begin{document}

\title[On virtual chirality of $3$-manifolds]{On virtual chirality of $3$-manifolds}

\author{Hongbin Sun}
\address{Department of Mathematics, Rutgers University - New Brunswick, Hill Center, Busch Campus, Piscataway, NJ 08854, USA}
\email{hongbin.sun@rutgers.edu}

\author{Zhongzi Wang}
\address{School of Mathematical Sciences, Peking University, Beijing 100871, CHINA}
\email{wangzz22@stu.pku.edu.cn}

\thanks{The first author is partially supported by the Simons Collaboration Grant 615229.}

\date{\today}
\begin{abstract} 

We prove that if a prime  3-manifold $M$ is not finitely covered by the 3-sphere or a product manifold, then $M$ is virtually chiral,
i.e. it has a finite cover that does not admit an orientation reversing self-homeomorphism.
In general if a $3$-manifold contains a virtually chiral prime summand, then it is virtually chiral.
\end{abstract}

\maketitle
\vspace{-.5cm}
\tableofcontents
\section{Introduction}

In this paper, we assume all $3$-manifolds are compact, connected, orientable, and have empty or tori boundary. All hyperbolic $3$-manifolds are complete and have finite volume. All homology groups are under the $\mathbb{Z}$-coefficient.

A compact, orientable manifold $M$ is said to be {\it chiral} if it does not admit an orientation reversing self-homeomorphism, and $M$ is said to be {\it achiral} if it admits an orientation reversing self-homeomorphism. The chirality is a fundamental property of an orientable manifold, and it is especially more interesting in dimension $3$.
Determining the chirality of a manifold  is often tricky. Here are a few famous examples: the complement of the figure-eight  knot  is achiral (1844,  Lusting), and the complement of the trefoil knot is chiral (1914, Dehn). For rather recent papers related to chirality, see \cite{BL}, \cite{Mu}, \cite{SWWZ} for example.

We say that a manifold $M$ {\it virtually} has a certain property if $M$ admits a finite cover $\widetilde{M}$ that has this property. The study of virtual properties of $3$-manifolds has always been an active research topic on $3$-manifolds, especially after Agol's solution (\cite{Agol}) of Thurston's virtual Haken conjecture on hyperbolic $3$-manifolds, which is based on Wise's machinery on cube complexes (\cite{Wise}). The general philosophy is: a $3$-manifold has a lot of finite covers, so it should have a finite cover that satisfies any given reasonable property. Various results in this direction after  \cite{Agol} and \cite{Wise} can be found in  a survey \cite{LS}.

In this paper, we study virtually chirality of $3$-manifolds, i.e. whether a given $3$-manifold $M$ has a finite cover $\widetilde{M}$ that does not admit an orientation reversing self-homeomorphism. In general, if a $3$-manifold $M$ admits an orientation reversing self-homeomorphism $f:M\to M$, $f$ may not be lifted to a finite cover $\widetilde{M}$ of $M$. Moreover, if the (orientable) minimal orbifold of a non-arithmetic hyperbolic $3$-manifold $M$ admits no orientation reversing self-isometry, then any $3$-manifold commensurable to $M$ is chiral. So it is expected that most $3$-manifolds should be virtually chiral, and the following result confirms this expectation.

\begin{theorem}\label{main}
Let $M$ be a compact, connected, orientable, prime $3$-manifold with empty or tori boundary.
If $M$ is not covered by the $S^3$ or the product of a surface and $S^1$, then $M$ has a finite cover that does not admit an orientation reversing self-homeomorphism.
\end{theorem}

Theorem \ref{main} deals with geometric $3$-manifolds  and non-geometric $3$-manifolds. 
There are   eight geometries in dimension $3$:
$$S^3, S^2\times \mathbb{E}^1, \mathbb{E}^3, \mathbb{H}^2\times \mathbb{E}^1, \text{Nil}, \widetilde{\text{PSL}(2,\mathbb{R})}, \text{Sol}, 
\mathbb{H}^3 .$$

3-manifolds supporting the first four geometries are finitely covered by either $S^3$ or a product manifold, so they are not  considered  in Theorem \ref{main}. 
Theorem \ref{main} indeed fails for the first four geometries. Actually, $S^3, S^2\times S^1, T^3, \Sigma_g\times S^1$, the most standard
representatives in the first four classes,  admits no finite chiral coverings.

It is known that Theorem \ref{main}  holds for  $\widetilde{\text{PSL}(2,\mathbb{R})}$ and $\text{Nil}$ $3$-manifolds. A very recent work  \cite{TWW} proved that every commensurable class of Sol 3-manifolds contains a chiral element, whose proof is based on number theory.
Theorem \ref{main} for Sol 3-manifolds follows  from \cite{TWW} (see Remark \ref{sol}). Indeed, \cite{TWW} inspires this work.

So the proof of Theorem \ref{main} divides into three case: hyperbolic $3$-manifolds,  mixed $3$-manifolds, and graph $3$-manifolds. The proof for mixed $3$-manifolds depends on the hyperbolic case, and the proof for graph $3$-manifolds is independent of the other cases. These three cases will be proved in Sections \ref{hyperbolic}, \ref{mixed}, and \ref{graph}, respectively.
In Section \ref{reducible}, based on Theorem \ref{main}, we will prove the following result on virtual chirality of reducible $3$-manifolds.

\begin{theorem}\label{main2}
Let $M$ be a compact, connected, orientable, prime $3$-manifold with empty or tori boundary.
If $M$ contains a virtually chiral prime summand, then $M$ is virtually chiral.
In particular,  $M$ is 
virtually chiral if it has non-vanishing simplicial volume.

\end{theorem}

Known facts used in our  proof for Theorem \ref{main} include: For complete finite volume hyperbolic $3$-manifold $M$, $\pi_1(M)$   is subgroup conjugacy distinguished \cite{CZ}, and $M$ has a finite cover $\widetilde M$ so that the  $\widetilde M/\partial \widetilde M$ has positive first Betti number  \cite{AFW}. (Here  $\widetilde M/\partial \widetilde M$ means that we pinch each component of $\partial \widetilde M$ to a point, instead of pinching $\partial \widetilde M$ to a single point.) 
Some constructions of finite coverings of mixed 3-manifolds \cite{DLW}, \cite{Liu}. All these facts rely on Agol's virtual specialization theorem (\cite{Agol}), 
 which is based on Wise work on cube complexes (\cite{Wise}), as well as Przytycki and Wise's work \cite{PW}.

In \cite{LR}, Long and Reid proved that for any non-arithmetic hyperbolic $n$-manifold $M$ with $n\geq 2$, there exists a manifold $M'$ commensurable to $M$ such that $\text{Isom}_+(M')$ is trivial, i.e. $M'$ only admits the trivial orientation preserving self-isometry. When restricting to hyperbolic $3$-manifolds, Theorem \ref{main} gives a finite cover $\widetilde{M}$ of $M$ that has no orientation reversing self-isometry, so these two results are complement of each other in some sense.
\bigskip

The authors are informed that Ye Tian and Hang Yin have proved that hyperbolic 3-manifolds are virtually chiral, using an argument in number theory,  and Alan Reid also makes progress for hyperbolic 3-manifolds in this direction.

\section{Preliminary}\label{preliminary}

In this section, we review some basic notions on decomposition of $3$-manifolds. More details on topology and geometry of $3$-manifolds can be found in \cite{Hat} and \cite{Sco} respectively.

For a compact, connected, orientable, and irreducible $3$-manifold $M$ with empty or tori boundary, the geometrization of $3$-manifolds implies the existence of a canonical family $\mathcal{T}$ of incompressible tori and Klein bottle in $M$, such that each component of $M\setminus \mathcal{T}$ (called a geometric piece or simply a piece of $M$) supports one of Thurston's eight geometries. 

When $\mathcal{T}$ is non-empty, we say that the $3$-manifold $M$ is a {\it non-geometric $3$-manifold}. In this case, each piece supports either the hyperbolic geometry or the $\mathbb{H}^2\times \mathbb{E}^1$ geometry, and we call it a hyperbolic piece or a Seifert piece respectively. For a Seifert piece $J\subset M$, there exists a compact orientable surface $F$ and a periodic self-homeomophism $f:F\to F$, such that $J$ is homeomorphic to the mapping torus $F\times I/(x,0)\sim (f(x),1)$.

The geometric decomposition is closely related to the JSJ decomposition of $3$-manifolds. These two decompositions are slightly different from each other, but also coincide in most of the cases. For any $3$-manifold $M$ that does not support the $\text{Sol}$ geometry, there exists a double cover $\widetilde{M}$ of $M$ that does not contain the twisted $I$-bundle over Klein bottle as a submanifold, and these two decompositions coincide with each other in this case. Since we work on virtual properties of $3$-manifolds, we will not distinguish the subtle difference of these two decompositions in this paper.

We call a non-geometric $3$-manifold $M$ a {\it mixed $3$-manifold} if it has at least one hyperbolic piece under its geometric decomposition.
We call a non-geometric $3$-manifold $M$ a {\it graph $3$-manifold} if all of its pieces are Seifert manifolds, i.e, it is not a mixed $3$-manifold. Note that Seifert $3$-manifolds and $\text{Sol}$ $3$-manifolds are not considered to be graph $3$-manifolds in this paper. For a graph $3$-manifold, all of its geometric pieces are Seifert $3$-manifolds whose base $2$-orbifold has a hyperbolic structure (and has negative Euler characteristic).

Let $M$ be a graph $3$-manifold that does not contain the twisted $I$-bundle over Klein bottle as its submanifold. Then any Seifert piece of $M$ is a circle bundle over a hyperbolic $2$-orbifold and has a unique Seifert structure. For each JSJ torus $T$ of $M$, it is adjacent to two (not necessarily distinct) Seifert pieces $J_1,J_2$ of $M$, and the regular fibers of these two Seifert pieces give two distinct slopes $s_1,s_2$ on $T$. We use $b_T\in \mathbb{Z}_{\geq 1}$ to denote the geometric intersection number of $s_1$ and $s_2$ on $T$.

If a Seifert piece $J$ of $M$ has no singular fiber, then $J$ is homeomorphic to a product $F\times S^1$, and not that such a product structure on $J$ is not unique. Let $T_1,\cdots, T_k$ be the boundary components of $J$ contained in the interior of $M$. Once we fix a product structure on $J$ and fix orientations of $F$ and $S^1$, we have oriented slopes $h_i,c_i$ on $T_i$. Here $h_i$ is the oriented regular fiber on $T_i$, and $c_i$ is the oriented boundary component of $F$ contained in $T_i$. Then $h_i$ and $c_i$ form a basis of $H_1(T_i;\mathbb{Z})\cong \mathbb{Z}^2$, and we give  $H_1(T_i;\mathbb{Z})$ a coordinate such that $h_i=(1,0)$ and $c_i=(0,1)$.

Let $J_i$ be the Seifert piece adjacent to $T_i$ other than $J$, let $f_i$ be the slope on $T_i$ given by the regular fiber of $J_i$. We give $f_i$ an orientation so that it has positive second coordinate. Then we have $f_i=(a_i,b_{T_i})$ for some $a_i\in \mathbb{Z}$ and $\text{gcd}(a_i,b_{T_i})=1$ holds. These notions will be used in Section \ref{graph}.

Now we make some orientation conventions which will be used in Sections \ref{mixed} and \ref{graph}. Let $M$ is compact connected oriented 3-manifold. 
Let $p: \widetilde M\to M$ be a finite covering. Then  $\widetilde M$ always has the induced orientation from $M$, 
i.e.  the covering degree of $p$ is positive.
Suppose further $M$ is irreducible and  has empty or tori boundary. If $M$ has a non-trivial geometric decomposition, then the covering
$p| : \widetilde J \to J$ has positive degree, where $J$ is a geometric piece of $M$, and $\widetilde J$ is a component of $p^{-1}(J)$,  $J$ has the induced  
orientation from $M$ and $\widetilde J$ has the induced orientation from $\widetilde M$.
 
\section{The case of hyperbolic $3$-manifolds}\label{hyperbolic}

We prove the hyperbolic case of Theorem \ref{main} in this section, and we restate the result here.

\begin{proposition}\label{hyperbolictheorem}
Let $M$ be an orientable, finite volume hyperbolic $3$-manifold, then $M$ has a finite cover $\widetilde{M}$ that is chiral.
\end{proposition}

We first review a little bit of hyperbolic geometry here. Since $M$ is an orientable, finite volume hyperbolic $3$-manifold, $\pi_1(M)$ can be identified with a discrete subgroup $$\pi_1(M)= G<\text{Isom}_+(\mathbb{H}^3)\cong \text{PSL}(2,\mathbb{C}).$$ For any loxodromic element $\gamma\in G<\text{PSL}(2,\mathbb{C})$, its trace $\text{tr}(\gamma)$ is a well-defined complex number up to multiplying $\pm1$, and actually $$\text{tr}(\gamma)\in (\mathbb{C}\setminus [-2,2])/\pm 1.$$ The conjugacy class of $\gamma$ corresponds to a closed geodesic $c$ in $M$. The complex length ${\bf l}(c)$ of $c$ is defined by $${\bf l}(c)=l(c)+i\theta(c) \in \mathbb{C}/2\pi i\mathbb{Z},$$ where $l(c)>0$ is the (real) length of $c$, and $\theta(c)\in \mathbb{R}/2\pi \mathbb{Z}$ is the rotation angle of the loxodromic isometry $\gamma$. The half length $\frac{{\bf l}(c)}{2}$ is well-defined up to adding $\pi i$, thus $\frac{{\bf l}(c)}{2}\in \mathbb{C}/\pi i\mathbb{Z}$. We have the equation 
\begin{align}\label{tracelength}
2\cosh(\frac{{\bf l}(c)}{2})=\text{tr}(\gamma),
\end{align}
 where both sides are complex numbers well-defined up to multiplying $\pm 1$. Equation (\ref{tracelength}) also implies that ${\bf l}(c)$ is uniquely determined by $\text{tr}(\gamma)$.

To prove Proposition \ref{hyperbolictheorem}, we need the following definition in group theory.

\begin{definition}\label{distinguished}
A group $G$ is {\it subgroup conjugacy distinguished} if for any finitely generated subgroup $H<G$ and any $g\in G$ that is not conjugate to an element $H$, there exists a homomorphism $\phi:G\to Q$ to a finite group $Q$, such that $\phi(g)$ is not conjugate to an element in $\phi(H)$.
\end{definition}

By using Agol's virtual specialization theorem in \cite{Agol} (based on Wise's work in \cite{Wise}), Chagas and Zalesskii proved the following result in \cite{CZ}.

\begin{theorem}\label{distinguishedtheorem}
The fundamental group $\pi_1(M)$ of a hyperbolic compact $3$-manifold $M$ (closed or with cusps) is subgroup conjugacy distinguished.
\end{theorem}

Now we are ready to prove Proposition \ref{hyperbolictheorem}.
\begin{proof}
Since $M$ is a finite volume hyperbolic $3$-manifold, by Theorem 3.3.7 of \cite{MR}, there exists a loxodromic element $\gamma\in G=\pi_1(M)$ such that $\text{tr}(\gamma^2)$ is a complex number with non-zero imaginary part. So neither the real part nor the imaginary part of $\text{tr}(\gamma)$ is zero, and $$\text{tr}(\gamma)\ne \overline{\text{tr}(\gamma)}\in (\mathbb{C}\setminus [-2,2])/\pm1,$$
where $\overline z$ is the complex conjugation of $z \in \mathbb{C}$.

Let $c$ be the closed geodesic corresponding to the conjugacy class of $\gamma$. Since $M$ has only finitely many closed geodesics of bounded length, there are only finitely many distinct closed geodesics $c_1,\cdots,c_n$ in $M$, such that a representative $\gamma_j\in G$ of the conjugacy class corresponding to $c_j$ satisfies $\text{tr}(\gamma_j)=\overline{\text{tr}(\gamma)}$ for all $j$. Thus any element in $G$ with trace $\overline{\text{tr}(\gamma)}$ is conjugate to one of $\gamma_1,\cdots, \gamma_n$.

Since $\text{tr}(\gamma_j)\ne \text{tr}(\gamma)$, we know that $\gamma_j$ is not conjugate to $\gamma^{\pm 1}$. Moreover, by equation (\ref{tracelength}), $\text{tr}(\gamma_j)=\overline{\text{tr}(\gamma)}$ implies that ${\bf l}(c_j)=\overline{{\bf l}(c)}$. So $l(c_j)=l(c)>0$, $l(c_j)\ne n\cdot l(c)$ for any integer $n\geq 2$, and $\gamma_j$ is not conjugate to $\gamma^{\pm n}$. Let $H$ be the cyclic subgroup of $G$ generated by $\gamma$, then for each $j$, $\gamma_j$ is not conjugate to an element in $H$.

By Theorem \ref{distinguishedtheorem}, for each $j$, there exists a homomorphism $$\phi_j:G\to Q_j$$ to some  finite group $Q_j$, such that $\phi_j(\gamma_j)$ is not conjugate to an element in $\phi_j(H)$. Then we have the following homomorphism
$$\phi=\oplus_{i=1}^n \phi_i:G\to Q=\oplus_{i=1}^n Q_i,$$
where $Q$ is a finite group. For any $j$, $\phi(\gamma_j)$ is not conjugate to an element in $\phi(H)$. 

Since $Q$ is a finite group, $\widetilde{G}=\phi^{-1}(\phi(H))$ is a finite index subgroup of $G$, and the conjugacy class of $\gamma_j$ in $G$ is disjoint from $\widetilde{G}$ for all $j$. So there is no element in $\widetilde{G}$ whose trace equals $\overline{\text{tr}(\gamma)}$. 

Let $\widetilde{M}$ be the finite cover of $M$ corresponding to 
$$\widetilde{G}<G=\pi_1(M)<\text{Isom}_+(\mathbb{H}^3).$$ 
Let $N(\widetilde G)$ be the normalizer of $\tilde G$ in $\text{Isom}(\mathbb{H}^3)$, then we have $$\text{Isom}(\widetilde M)=N(\widetilde G)/ \widetilde G.$$

Suppose there exists an orientation reversing self-homeomorphism $f:\widetilde{M}\to \widetilde{M}$, by the Mostow--Prasad rigidity, $f$ is homotopic to an orientation reversing self-isometry of $\widetilde{M}$ and we still denote it by $f$. Lifting $f$ to the universal cover $\mathbb{H}^3$ of $\widetilde{M}$, we get an orientation reversing isometry $\widetilde{f}\in \text{Isom}(\mathbb{H}^3)$ and it lies in the normalizer  $$N(\widetilde{G})<\text{Isom}(\mathbb{H}^3),$$ and in particular $\widetilde{f}\cdot \gamma \cdot \widetilde{f}^{-1}\in \widetilde{G}$.

Now we consider $\text{Isom}(\mathbb{H}^3)$ as the group of angle-preserving self-homeomorphisms of the Riemann sphere $\overline{\mathbb{C}}$. We can write $\widetilde{f}=h\cdot g$ with $h(z)=\overline{z}$ and $g\in \text{Isom}_+(\mathbb{H}^3)$. 
For  $\eta \in PSL_2(C)$, suppose  $$\eta(z)= \frac{az+b}{cz+d}.$$ Then 
 $$h\circ \eta \circ  h^{-1}(z)=h\circ \eta (\overline z)= h( \frac{a\overline z+b}{c\overline z+d})=\frac{ \overline a z+\overline b}{\overline c z+\overline d}.$$
So $$\text{tr}(h\cdot \eta\cdot h^{-1})
= \overline{\text{tr}(\eta)}.\qquad (3.2)$$

Then 
\begin{align*}
\text{tr}(\widetilde{f}\cdot \gamma \cdot \widetilde{f}^{-1})=\text{tr}(h\cdot (g\cdot \gamma \cdot g^{-1})\cdot h^{-1})
= \overline{\text{tr}(g\cdot \gamma \cdot g^{-1})}=\overline{\text{tr}(\gamma)}.
\end{align*}
Here the second equality holds by (3.2), the third equality holds since $g\in \text{Isom}_+(\mathbb{H}^3)$.

Now we get an element $\widetilde{f}\cdot \gamma\cdot \widetilde{f}^{-1}\in \widetilde{G}$ whose trace equals $\overline{\text{tr}(\gamma)}$, which contradicts with our choice of $\widetilde{G}$. So $\widetilde{M}$ admits no orientation reversing self-homeomorphism, and the proof is done.
\end{proof}

\begin{remark}\label{sol} 
We end this section by giving more explanations  on Theorem \ref{main} for three other geometries,
especially for Sol  3-manifolds.

Theorem \ref{main} holds for $\widetilde{\text{PSL}(2,\mathbb{R})}$ and $\text{Nil}$ 3-manifolds, because $\widetilde{\text{PSL}(2,\mathbb{R})}$ and $\text{Nil}$ $3$-manifolds do not admit degree $-1$ self-maps (e.g. see Theorems 1.3 and 1.4 of \cite{SWWZ}, also Lemma 3.4 of \cite{NWW}).

Theorem \ref{main} for Sol 3-manifolds  is contained in  Corollary 3.9 of \cite{TWW} and its proof, as explain below.
Each  orientable $\text{Sol}$ $3$-manifold $M$ is covered by $M_\phi$, a torus bundle over the circle whose monodromy is a matrix $\phi\in \text{SL}(2,\mathbb{Z})$ with $|\text{tr}(\phi)|>2$.
Since we work on virtual properties of $3$-manifolds, we may assume that $M=M_{\phi}$.
In the proof of Corollary 3.9 of \cite{TWW}, it is indeed proved that any $\text{Sol}$ $3$-manifold $M_\phi$ is commensurable to a chiral $3$-manifold $M_\psi$. Moreover
 $M_\psi$ and $M_{\psi^n}$ have the same chirality for $n\ne 0$ (based on \cite[Theorem 3.7 and Lemma 3.3 (iii)]{TWW});
and  for some $m, n \ne 0$, there is a finite cover $M_{\psi^n}\to M_{\phi^m}$ (based on \cite[Lemma 2.6 (1)]{TWW}. 
Then we find a finite cover $M_{\psi^n}\to M_{\phi^m}\to M_{\phi}$ and $M_{\psi^n}$ is chiral.
\end{remark}

\section{The case of mixed $3$-manifolds}\label{mixed}

We will prove the mixed $3$-manifold case of Theorem \ref{main} in this section, and the hyperbolic case in Section \ref{hyperbolic} will play an essential role. We restate the result here.

\begin{proposition}\label{mixedtheorem}
Let $M$ be a compact, connected, orientable, irreducible $3$-manifold with empty or tori boundary. Suppose $M$ is a mixed $3$-manifold, then $M$ has a finite cover $\widetilde{M}$ that is chiral.
\end{proposition}

In this and the next sections, we need to construct  finite covers of a non-geometric $3$-manifold $M$ with required properties.
To construct a finite cover of a non-geometric $3$-manifold $M$, we usually start with finite covers of one or more geometric pieces of $M$, and want to construct a finite cover based on the given data. We recall two such constructions. Both of them rely on Agol's virtual specialization theorem (\cite{Agol}), which is based on Wise work on cube complexes (\cite{Wise}), as well as Przytycki and Wise's work \cite{PW}.

The first construction is Proposition 4.2 of \cite{DLW}, which has some control on all geometric pieces of the finite cover. We only state a slightly weaker result here, which is good enough for our applications.
\begin{proposition}\label{DLWconstruction}
Let $M$ be a compact, orientable, irreducible $3$-manifold with empty or tori boundary. Suppose $\hat{J}_1,\cdots, \hat{J}_s$ are finite covers of all JSJ pieces $J_1, . . . , J_s$ of M, respectively. Then there is a finite cover $M'$ of $M$, such that any elevation $J_i'$ of a JSJ piece $J_i$ in $M'$ is a cover of $J_i$ that factors through $\hat{J}_i$.
\end{proposition}
Here an elevation of $J_i$ in $M'$ means a connected component of the preimage of $J_i$ in $M'$.

The second construction is Lemma 6.2 of \cite{Liu}, which has a precise control on one geometric piece of the finite cover. We only state the result in some special cases, which avoids some undefined terminology.
\begin{lemma}\label{Liuconstruction}
Let $M$ be a compact, orientable, irreducible $3$-manifold with empty or tori boundary. If $J'$ is a connected finite cover of a geometric piece or a JSJ torus of $M$, then $J'$ has an embedded lift into a finite cover $M'$ of $M$.
\end{lemma}

The following quantity will play an important role in this section. 

Let $N$ be a compact, connected $3$-manifold with boundary, we define 
\begin{align}\label{c}
c(N)=\text{rank}_{\mathbb{Z}}\big(\text{coker}(H_1(\partial N)\to H_1(N))\big),
\end{align}
where $\text{coker}(H_1(\partial N)\to H_1(N))=H_1(N)/ \text{im}(H_1(\partial N))$.

Item (C.13) of \cite{AFW} implies the following result on virtual positivity of $c(N)$, which relies on Agol's virtual specialization theorem (\cite{Agol}).
\begin{proposition}\label{positivec}
Let $N$ be a finite volume complete hyperbolic $3$-manifold with cusps, then for any $k\in \mathbb{N}$, $N$ has a finite cover $N'$, such that $c(N')\geq k$.
\end{proposition}

For the sake of our application, we only care about the positivity of $c(N)$. If $c(N)>0$ holds, then there is a surjective homomorphism $H_1(N)\to \mathbb{Z}$ that vanishes on the image of $H_1(\partial N)\to H_1(N)$. Note that if we have a finite cover $N'$ of $N$, then $c(N)>0$ implies $c(N')>0$.

Now we are ready to prove Proposition \ref{mixedtheorem}.
\begin{proof}
The proof consists of a sequence of constructions of finite covers of the mixed $3$-manifold $M$.

{\bf Step I}. We claim that $M$ has a finite cover $M'$, such that any hyperbolic piece $J'$ of $M'$ satisfies $c(J')>0$. 

Let $J_{1,\text{hyp}},\cdots, J_{r,\text{hyp}}$ be all hyperbolic pieces of $M$, with $r\geq 1$ holds, and let $J_{1,\text{Sei}},\cdots, J_{s,\text{Sei}}$ be all Seifert pieces of $M$. By Proposition \ref{positivec}, each hyperbolic piece $J_{i,\text{hyp}}$ has a finite cover $\hat{J}_{i,\text{hyp}}$ such that $c(\hat{J}_{i,\text{hyp}})>0$. For each Seifert piece $J_{i,\text{Sei}}$, we take $\hat{J}_{i,\text{Sei}}=J_{i,\text{Sei}}$. Then Proposition \ref{DLWconstruction} provides us a finite cover $M'$ of $M$, such that any hyperbolic piece $J'$ of $M'$ is a finite cover of some $\hat{J}_{i,\text{hyp}}$, so it satisfies $c(J')>0$

{\bf Step II}. We claim that $M'$ has a finite cover $M''$, such that any hyperbolic piece $J''$ of $M''$ satisfies $c(J'')>0$, and it has a hyperbolic piece $J''_{0,\text{hyp}}$ that admits no orientation reversing self-homeomorphism.

Take a hyperbolic piece $J_{0,\text{hyp}}'$ of $M'$. By Proposition \ref{hyperbolictheorem}, $J_{0,\text{hyp}}'$ has a finite cover $J''_{0,\text{hyp}}$ that admits no orientation reversing self-homeomorphism. By Lemma \ref{Liuconstruction}, $J_{0,\text{hyp}}''$ embedded lifts into a finite cover $M''$ of $M'$. Since any hyperbolic piece $J'$ of $M'$ satisfies $c(J')>0$, so does any hyperbolic piece $J''$ of $M''$.

{\bf Step III}. We finish the proof in this step.

Let $J_{0,\text{hyp}}'',J_{1,\text{hyp}}'',\cdots, J_{r,\text{hyp}}''$ be all hyperbolic pieces of $M''$, and let $J_{1,\text{Sei}}'',\cdots,J_{s,\text{Sei}}'' $ be all Seifert pieces of $M''$. For any $i=1,\cdots, r$, since $c(J_{i,\text{hyp}}'')>0$, there exists a surjective homomorphism 
$$\phi_i:H_1(J_{i,\text{hyp}}'')\to \mathbb{Z}$$ that vanishes on $$\text{im}\big(H_1(\partial J_{i,\text{hyp}}'')\to H_1(J_{i,\text{hyp}}'')\big).$$

By the M-V sequence, we know that $H_1(M'')$ is isomorphic to 
$$\text{im}\Big(H_1(J_{0,\text{hyp}}'')\oplus (\oplus_{i=1}^rH_1(J_{i,\text{hyp}}''))\oplus(\oplus_{j=1}^s H_1(J_{j,\text{Sei}}''))\to H_1(M'')\Big)\oplus \mathbb{Z}^{1-\chi},$$ 
where the homomorphism is induced by inclusions of submanifolds, and $\chi$ is the Euler characteristic of the dual graph of $M''$. 

Take a positive integer $k$, so that $$k\cdot \text{vol}(J_{i,\text{hyp}}'')>\text{vol}(J_{0,\text{hyp}}'')$$ for all $i=1,\cdots,r$. We first construct a surjective homomorphism $$\phi:H_1(J_{0,\text{hyp}}'')\oplus (\oplus_{i=1}^rH_1(J_{i,\text{hyp}}''))\oplus(\oplus_{j=1}^s H_1(J_{j,\text{Sei}}''))\to \mathbb{Z}/k\mathbb{Z}.$$ On $H_1(J_{0,\text{hyp}}'')$ and $H_1(J_{j,\text{Sei}}'')$, we take the trivial homomorphism. On $H_1(J_{i,\text{hyp}}'')$, we take $\phi$ to be the composition $$H_1(J_{i,\text{hyp}}'')\xrightarrow{\phi_i}\mathbb{Z}\to \mathbb{Z}/k\mathbb{Z},$$ where the second homomorphism is the natural surjection. By the M-V sequence, the kernel of $$H_1(J_{0,\text{hyp}}'')\oplus (\oplus_{i=1}^rH_1(J_{i,\text{hyp}}''))\oplus(\oplus_{j=1}^s H_1(J_{j,\text{Sei}}''))\to H_1(M'')$$ is contained in the image of  $$H_1(\partial J_{0,\text{hyp}}'')\oplus (\oplus_{i=1}^rH_1(\partial J_{i,\text{hyp}}''))\oplus(\oplus_{j=1}^s H_1(\partial J_{j,\text{Sei}}''))$$ (under the natural inclusion). So $\phi$ induces a surjective homomorphism $$\psi':\text{im}\Big(H_1(J_{0,\text{hyp}}'')\oplus (\oplus_{i=1}^rH_1(J_{i,\text{hyp}}''))\oplus(\oplus_{j=1}^s H_1(J_{j,\text{Sei}}''))\to H_1(M'')\Big)\to \mathbb{Z}/k\mathbb{Z},$$ 
and it extends to a surjective homomorphism $$\psi:H_1(M'')\to \mathbb{Z}/k\mathbb{Z}$$ by taking the trivial homomorphism on the $\mathbb{Z}^{1-\chi}$ component.

Let  $\widetilde p: \widetilde{M}\to M''$ be the $k$-sheet cyclic cover  corresponding to the kernel of $$\pi_1(M'')\to H_1(M'')\xrightarrow{\psi}\mathbb{Z}/k\mathbb{Z}.$$ By our construction of $\psi$, each geometric piece of $\widetilde{M}$ is homeomorphic to one of the following, as oriented manifolds:
\begin{enumerate}
\item a $k$-sheet cyclic cover of the hyperbolic $3$-manifold $J_{i,\text{hyp}}''$, $i=1,\cdots,r$,
\item the Seifert manifold $J_{j,\text{Sei}}''$, $j=1,\cdots,s$.
\item the hyperbolic $3$-manifold $J_{0,\text{hyp}}''$.
\end{enumerate}

Now we suppose that there is an  orientation reversing self-homeomorphism $ \widetilde h: \widetilde{M}\to \widetilde{M}$. 
Let $\widetilde J_1$ be a geometric piece of $\widetilde{M}$ with type $J_{0,\text{hyp}}''$.
By the uniqueness of geometric decomposition, up to isotopy, there is a geometric piece $\widetilde J_2$ of $\widetilde{M}$
such that $\widetilde{h}|: \widetilde J_1\to \widetilde J_2$ is a homeomorphism. For the induced orientations of $ \widetilde J_1$ and $ \widetilde J_2$ from $ \widetilde M$ , $ \widetilde h|$ is orientation reversing.
  Item (1) is not homeomophic to $-J_{0,\text{hyp}}''$ since it has volume $k\cdot \text{vol}(J_{i,\text{hyp}}'')>\text{vol}(J_{0,\text{hyp}}'')$, item (2) is not homeomorphic to $-J_{0,\text{hyp}}''$ since it is a Seifert manifold.

  So $\widetilde{J}_2$ must be in item (3).
By the orientation convention at the end of Section 2, $\widetilde p_i=\widetilde p| : \widetilde J_i\to   J_{0,\text{hyp}}''$ is an orientation
 preserving homeomorphism, $i=1,2$. Then 
 $$\widetilde{p}_2\circ \widetilde{h}|\circ \widetilde{p}_1^{-1}:J_{0,\text{hyp}}''\to \widetilde{J}_1\to \widetilde{J}_2\to J_{0,\text{hyp}}''$$
 is an orientation preserving self-homeomorphism of $J_{0,\text{hyp}}''$. It is impossible due to our choice of $J_{0,\text{hyp}}''$ in step II, which is chiral.
 
So $\widetilde{M}$ admits no orientation reversing self-homeomorphism, as desired.
\end{proof}

\bigskip
\bigskip

\section{The case of graph $3$-manifolds}\label{graph}

We will prove the graph $3$-manifold case of Theorem \ref{main} in this section, and the proof does not rely on the main results in previous sections. We restate the result here.

\begin{proposition}\label{graphtheorem}
Let $M$ be a compact, connected, orientable, irreducible $3$-manifold with empty or tori boundary. Suppose $M$ is a graph $3$-manifold, then $M$ has a finite cover $\widetilde{M}$ that is chiral.
\end{proposition}

\begin{proof}
Again, the proof consists of a sequence of constructions of finite covers of the graph $3$-manifold $M$.

{\bf Step I}. We claim that $M$ has a finite cover $M'$, such that each Seifert piece $J'$ of $M'$ is homeomorphic to $\Sigma\times S^1$, where $\Sigma$ is a orientable surface with genus at least $2$.

At first, we pass to a double cover of $M$ and still denote it by $M$, so that it does not contain the twisted $I$-bundle over Klein bottle as a submanifold. Then the JSJ and geometric decompositions of $M$ are the same. The base orbifold of any Seifert piece $J$ of $M$ has a hyperbolic structure. So $J$ has a finite cover that is a circle bundle over an orientable hyperbolic surface with boundary, and we take a further finite cover $\hat{J}$ to make sure the base surface has genus at least $2$. 

Then we apply Proposition \ref{DLWconstruction} to get a finite cover $M'$ of $M$, so that each Seifert piece $J'$ of $M'$ is a finite cover of some $\hat{J}$. So $J'$ is the product of an orientable surface with genus at least $2$ and $S^1$, and this property holds when passing to any further finite cover of $M$.

{\bf Step II}. We claim that $M'$ has a finite cover $M''$ that contains a JSJ torus $T''\subset M''$ with $b_{T''}\geq 3$.

Take a JSJ torus $T'$ of $M'$, and we can assume $b_{T'}=1$ or $2$, otherwise we simply take $M''=M'$ and $T''=T'$. We denote regular fibers of two adjacent Seifert pieces of $T'$ by $s_1'$ and $s_2'$. We give a basis of $H_1(T')\cong \mathbb{Z}\oplus \mathbb{Z}$, so that $s_1'$ corresponds to $(1,0)$ and $s_2'$ corresponds to $(k,1)$ or $(k,2)$ (corresponding to the case of $b_{T'}=1$ or $2$ respectively). 

If $k$ has $3$ as a factor, we can re-choose the second vector in the basis of $H_1(T')\cong \mathbb{Z}\oplus \mathbb{Z}$, so that $k$ is coprime with $3$. We take $T''$ to be the finite cover of $T'$ corresponding to $3\mathbb{Z}\oplus \mathbb{Z}<\mathbb{Z}\oplus \mathbb{Z}$. It is easy to check that any pair of elevations of $s_1'$ and $s_2'$ in $T''$ have intersection number $3$ or $6$.

Now we apply Lemma \ref{Liuconstruction} to the finite cover $T''\to T'\subset M'$ to obtain a finite cover $M''$ of $M'$, so that $T''$ has an embedded lift into $M''$. On $T''\subset M''$, the regular fibers of its adjacent pieces are elevations of $s_1'$ and $s_2'$, so we have $b_{T''}=3$ or $6$. In particular $b_{T''}\geq 3$ holds.

{\bf Step III}. We claim that $M''$ has a finite cover $\bar{M}$ containing a Seifert piece $\bar{J}_0\subset \bar{M}$ such that the following condition holds. Let $s_1,\cdots,s_k$ be the (unoriented) slopes on $\partial \bar{J}_0$ given by adjacent Seifert pieces of $\bar{M}$, there is no orientation reversing self-homeomorphism $f:\bar{J}_0\to \bar{J}_0$ that preserves the set of slopes $\{s_1,\cdots,s_k\}$.

Take a Seifert piece $J_0''$ of $M''$ that contains $T_0''=T''$ as a boundary component, and $J_0''$ is homeomorphic to $F_0\times S^1$. Let $c_0, c_1,\cdots,c_l$ be the boundary components of $F$, such that $T_i''=c_i\times S^1$, $i=0,1,\cdots,l$, lie in the interior of $M''$. With respect to the coordinate of $H_1(T_i'')$ described in Section \ref{preliminary}, let $(a_i,b_{T_i''})$ be the slope $s_i''$ on $T_i''$ given by the regular fiber of the Seifert piece adjacent to $J_0''$ along $T_i''$. By Step II, we have $b_{T_0''}=b_{T''}\geq 3$.

Let $F_0'$ be the $2$-orbifold obtained from $F_0$ by pasting a (regular) disk along $c_0$ and pasting a singular disc with singular index $b_{T_i''}$ along $c_i$. We consider $F_0$ as a subsurface of $F_0'$, take a surface finite cover $q:S_0'\to F_0'$ of $F_0'$, and let $\bar{S}_0=q^{-1}(F_0)$. Then we get a Seifert manifold $\bar{J}_0=\bar{S}_0\times S^1$ that is a finite cover of $J''_0=F_0\times S^1$, such that each elevation of $s_0''$ has coordinate $(a_0,b_{T_0''})$, and each elevation of $s_i''$ has coordinate $(a_i,1)$ for $i=1,\cdots,l$. For simplicity, we use $b_0$ to denote the integer $b_{T_0''}\geq 3$.

Now we apply Lemma \ref{Liuconstruction} to the finite cover $\bar{J}_0\to J''_0\subset M''$ to get a finite cover $\bar{M}\to M''$, such that $\bar{J}_0$ has an embedded lift in $\bar{M}$. By our construction of $\bar{J}_0$, for any JSJ torus of $\bar{J}_0$, the regular fiber of the adjacent JSJ piece has coordinate either $(a_0,b_0)$ or $(a_i,1)$ for some $i=1,\cdots,l$.

Suppose there exists an orientation reversing self-homeomorphism $f:\bar{J}_0\to \bar{J}_0$ that preserves the slopes given by adjacent regular fibers. Up to isotopy, we can assume that $f$ preserves the fibers of $\bar{J}_0$. Let $\bar{T}\subset \bar{J}_0$ be a JSJ torus that contains a slope with coordinate $(a_0,b_0)$. The restriction of $f$  is an orientation reversing homeomorphism $f|_{\bar{T}}:\bar{T}\to f(\bar{T})$, where both $\bar{T}$ and $f(\bar{T})$ have induced orientations from $\bar{J}_0$. Then $f|_{\bar{T}}$ satisfies one of the following conditions:
\begin{itemize}
\item either $(f|_{\bar{T}})_*(1,0)=(1,0)$ and $(f|_{\bar{T}})_*(0,1)=(m,-1)$ for some $m\in \mathbb{Z}$,
\item or $(f|_{\bar{T}})_*(1,0)=(-1,0)$ and $(f|_{\bar{T}})_*(0,1)=(-m,1)$ for some $m\in \mathbb{Z}$.
\end{itemize}

Then we have $(f|_{\bar{T}})_*(a_0,b_0)=\pm(a_0+mb_0,-b_0)$. As unoriented slopes, neither of them equals $\pm (a_i,1)$, since $b_0\geq 3$. Moreover, if $\pm(a_0+mb_0,-b_0)=\pm(a_0,b_0)$, then we have $mb_0=-2a_0$, which is impossible  since $\text{gcd}(a_0,b_0)=1$ and $b_0\geq 3$. 

So such an orientation reversing self-homeomorphism $f:\bar{J}_0\to \bar{J}_0$ does not exist, and the proof of Step III is done.

{\bf Step IV}. We will finish the proof in this step.

Let $\bar{J}_0,\bar{J}_1,\cdots,\bar{J}_n$ be all JSJ pieces of $\bar{M}$, with $\bar{J}_i=F_i\times S^1$. Let $g_i$ be the genus of $F_i$ with $g_i\geq 2$ holds. Take a positive integer $k$ so that $k(g_i-1)> g_0-1$ holds for all $i=1 ,\cdots, n$.

For each $i=1,\cdots,n$, we take a surjective homomorphism $\phi_i:H_1(\bar{J}_i)\to H_1(F_i)\to \mathbb{Z}$. Here the first homomorphism is induced by the bundle projection, and the second homomorphism is the dual of a non-separating simple closed curve on $F_i$. So $\phi_i$ vanishes on $\text{im}(H_1(\partial \bar{J}_i)\to H_1(\bar{J}_i))$.

Now we work as in Step III of the proof of Proposition \ref{mixedtheorem}, to construct a $k$-sheet cyclic cover $\widetilde{M}$ of $\bar{M}$, such that each JSJ piece $\widetilde{J}$ of $\widetilde{M}$ satisfies one of the following conditions. 
\begin{enumerate}
\item $\widetilde{J}$ is either homeomorphic to $\bar{J}_0$ as oriented manifolds, and the homeomorphism sends regular fibers of adjacent Seifert pieces to the slopes $\{s_1,\cdots,s_k\}$ in step III.
\item or homeomorphic to a $k$-sheet cyclic cover of $\bar{J}_i$, whose base surface has genus $k(g_i-1)+1$.
\end{enumerate}

Now we suppose that $\widetilde{M}$ has an orientation reversing self-homeomorphism. By the uniqueness of geometric decomposition, there is an orientation reversing homeomorphism from $\bar{J}_0$ to one of the two items above, as oriented manifolds, that sends $\{s_1,\cdots,s_k\}$ to regular fibers of adjacent JSJ pieces. However, item (1) fails by the conclusion of step III , and item (2) fails since  $k(g_i-1)+1>g_0$ and their base surfaces have different genera.

So $\widetilde{M}$ admits no orientation reversing self-homeomorphism, as desired.
\end{proof}

\bigskip

\section{On reducible $3$-manifolds}\label{reducible}

Theorem \ref{main} concerns virtual chirality of irreducible $3$-manifolds, and we consider reducible $3$-manifolds in this section. Basically, if a $3$-manifold contains a virtually chiral prime summand, then it is virtually chiral. Our result is stated in the following theorem.

\begin{theorem}\label{reducibletheorem}
Let $M$ be a compact, connected, orientable $3$-manifold with empty or tori boundary, and $M$ has a prime summand $P$ that either supports one of the following geometries: hyperbolic, $\text{Sol}$, $\text{Nil}$, $\widetilde{\text{PSL}(2,\mathbb{R})}$, or is non-geometric.
Then $M$ has a finite cover $\widetilde{M}$ that is chiral.
\end{theorem}

We first give an easy construction of finite covers of reducible $3$-manifolds.

\begin{lemma}\label{reduciblecover}
Let $M$ be a compact, connected, oriented $3$-manifold with oriented prime decomposition $M=\#_{i=1}^nP_i$. For each $i$, let $\widetilde{P}_i$ be a finite cover of $P_i$ with induced orientation. Then $M$ has a finite cover $\widetilde{M}$, such that each of its prime summand is homeomorphic to some $\widetilde{P}_i$ or $S^2\times S^1$ as oriented manifold, and each $\widetilde{P}_i$ shows up at least once.
\end{lemma}

\begin{proof}
We first take an explicit model of $M$. Let $\mathring{P}_1$ be the complement of $n-1$ disjoint open $3$-balls in $P_1$, and let the boundaries of these $3$-balls be $S_2,\cdots,S_n$. For each $i=2,\cdots,n$, let $\mathring{P}_i$ be the complement of one open $3$-ball in $P_i$, and let the boundary of this $3$-ball be $S'_i$. Then $M$ is homeomorphic to the manifold obtained from $\bigsqcup_{i=1}^n \mathring{P}_i$ by pasting $S_i$ and $S_i'$ via an orientation reversing homeomorphism.

Let $d_i$ be the covering degree of $\widetilde{P}_i\to P_i$, let $D=\text{lcm}(d_1,\cdots,d_n)$ and let $a_i=D/d_i$. Let $\mathring{\widetilde{P}}_i$ be the preimage of $\mathring{P}_i$ in $\widetilde{P}_i$. For $i=1$ and any $j=2,\cdots,n$, $\mathring{\widetilde{P}}_1$ has $d_1$ spherical boundary components that cover $S_j$. For $i=2,\cdots,n$, $\mathring{\widetilde{P}}_i$ has $d_i$ spherical boundary components that cover $S_i'$.

For each $i$, we take $a_i$ copies of $\mathring{\widetilde{P}}_i$ and take their disjoint union $\bigsqcup_{i=1}^na_i\mathring{\widetilde{P}}_i$. Then for each sphere $S_i$ or $S_i'$, there are exactly $a_1\cdot d_1=a_i\cdot d_i=D$ boundary spheres in $\bigsqcup_{i=1}^na_i\mathring{\widetilde{P}}_i$ that are mapped to it. We take an arbitrary pairing of these boundary spheres corresponding to $S_i$ and $S_i'$, and paste each pair by an orientation reversing homeomorphism to get a $3$-manifold $\widetilde{M}$.

Then $\widetilde{M}$ is a degree-$D$ cover of $M$, and it is a connected sum of $a_i$ copies of $\widetilde{P}_i$, and $1+(n-1)D-\sum_{i=1}^na_i$ copies of $S^2\times S^1$ (since the dual graph of $\widetilde{M}$ may not be a tree).

\end{proof}

\begin{lemma}\label{distinctcover}
Let $P$ be $3$-manifold either supporting one of the following geometries: hyperbolic, $\text{Sol}$, $\text{Nil}$, $\widetilde{\text{PSL}(2,\mathbb{R})}$, or is non-geometric. Then $P$ has a finite cover $\widetilde{P}$ that is not homeomorphic to $P$
\end{lemma}

\begin{proof}
We run a case-by-case argument.
\begin{itemize}
\item If $P$ is a hyperbolic $3$-manifold, we take a finite cover $\widetilde{P}$ of $P$ with degree $d>1$. For the hyperbolic volume, we have $\text{vol}(\widetilde{P})=d\text{vol}(P)>\text{vol}(P)$, so $\widetilde{P}$ is not homeomorphic to $P$.

\item If $P$ is a $\text{Sol}$ $3$-manifold, up to taking a double cover of $P$, we can assume that $P=M_{\phi}$ for some $\phi\in \text{SL}(2,\mathbb{Z})$ with eigenvalues $\lambda,\lambda^{-1}$ and $\lambda>1$ holds. Then $M_{\phi^k}$ is a finite cover of $P$, and $$|\text{Tor}(H_1(M_{\phi^k}))|=(\lambda^k-1)(1-\lambda^{-k})=\lambda^k+\lambda^{-k}-2.$$ So if $k$ is large enough, we have $|\text{Tor}(H_1(M_{\phi^k}))|>|\text{Tor}(H_1(P))|$. So $\widetilde{P}=M_{\phi^k}$ is a finite cover of $P$ and is not homeomorphic to $P$.

\item If $P$ is a $\text{Nil}$ $3$-manifold, up to taking a finite cover, $P$ is homeomorphic to a torus bundle $M_{\phi}$ with monodromy matrix $\phi=\pm
\begin{pmatrix}
1 & n \\
0 & 1
\end{pmatrix}$. If $k$ is large enough, $M_{\phi^{2k}}$ is a finite cover of $P$, and $|\text{Tor}(H_1(M_{\phi^k}))|=2|kn|>|\text{Tor}(H_1(P)|$. So $\widetilde{P}=M_{\phi^{2k}}$ is a finite cover of $P$ and is not homeomorphic to $P$.

\item If $P$ is a $\widetilde{\text{PSL}(2,\mathbb{R})}$ $3$-manifold, let the base $2$-orbifold of $P$ be denoted by $\mathcal{O}$. Then $\mathcal{O}$ has a hyperbolic structure, and we take a finite cover $\widetilde{\mathcal{O}}$ of $\mathcal{O}$ with degree $d>1$. Then $\chi(\widetilde{\mathcal{O}})=d\cdot \chi(\mathcal{O})<\chi(\mathcal{O})$, thus $\widetilde{\mathcal{O}}$ and $\mathcal{O}$ are not homeomorphic to each other. Let $\widetilde{P}$ be the pull-back of the bundle $P\to \mathcal{O}$ via the covering map $\widetilde{\mathcal{O}}\to \mathcal{O}$. Then $\widetilde{P}$ is a finite cover of $P$. $\widetilde{P}$ is not homeomorphic to $P$ because of the uniqueness of Seifert fibration, and their base orbifolds have different Euler characteristics.

\item  If $P$ is a mixed $3$-manifold, we take a finite cover $\widetilde{P}$ of $P$ with degree $d>1$. For the simplicial volume, we have $\|\widetilde{P}\|=d\|P\|>\|P\|$, so $\widetilde{P}$ is not homeomorphic to $P$.

\item If $P$ is a graph $3$-manifold, then it has a Seifert piece $J$ with hyperbolic base orbifold $\mathcal{O}$. Let $X$ be the minimum of Euler characteristics of base $2$-orbifolds of Seifert pieces of $P$. Then $\mathcal{O}$ has a finite degree surface cover $F$ with $\chi(F)<X$, and $\widetilde{J}=F\times S^1$ is a finite cover of $J$. By Lemma \ref{Liuconstruction}, $P$ has a finite cover $\widetilde{P}$ that contains $\widetilde{J}$ as a Seifert piece. Then $\widetilde{P}$ is not homeomorphic to $P$, since no Seifert piece of $P$ is homeomorphic to $\widetilde{J}\subset \widetilde{P}$, due to the consideration of Euler characteristics of base orbifolds.
\end{itemize}
\end{proof}

Now we are ready to prove Theorem \ref{reducibletheorem}.
\begin{proof}
We fix an orientation of $M$, and let $P$ be an oriented prime summand of $M$ that either supports one of the following geometries: hyperbolic, $\text{Sol}$, $\text{Nil}$, $\widetilde{\text{PSL}(2,\mathbb{R})}$, or is non-geometric. 

{\bf Step I}. We claim that $M$ has a finite cover $M'$, such that it has a chiral prime summand $P'$ that is a finite cover of $P$.

Let $M=a_0P\# a_1P_1\#\cdots \#a_n P_n$ be the prime decomposition of $M$. By Theorem \ref{main}, $P$ has a chiral finite cover $P'$. To apply Lemma \ref{reduciblecover}, we take the finite cover $P'$ for $P$, and take the trivial finite cover for $P_1,\cdots,P_n$. Then Lemma \ref{reduciblecover} provides a finite cover $M'$ of $M$, such that each of its prime summand is homeomorphic to $P'$, $P_1,\cdots,P_n$ or $S^2\times S^1$. So the proof of Step I is done.

{\bf Step II}. We finish the proof in this step. 

If $M'$ does not contain $-P'$ as its oriented prime summand, since $P'$ is chiral, it is not an oriented prime summand of $-M'$. By the uniqueness of prime decomposition of oriented $3$-manifolds, $M'$ is not homeomorphic to $-M'$. So $M'$ does not admit an orientation reversing self-homeomorphism, and we take $\widetilde{M}=M'$.

Now we assume that $-P'$ is an oriented prime summand of $M'$. Let the oriented prime decomposition of $M'$ be 
$$M'=kP'\#l(-P')\#a_1P_1'\#\cdots \# a_nP_n'.$$ 
Here each $P_i'$ is not homeomorphic to $P'$ or $-P'$. By Lemma \ref{distinctcover}, $-P'$ has a finite cover $\widetilde{(-P')}$ that is not homeomorphic to $-P'$. To apply Lemma \ref{reduciblecover}, we take the trivial finite cover for $P'$ and $P_1',\cdots, P_n'$, and take the finite cover $\widetilde{(-P')}$ for $(-P')$. Then Lemma \ref{reduciblecover} provides a finite cover $\widetilde{M}$ of $M'$, such that each of its oriented prime summand is homeomorphic to $P', \widetilde{(-P')},P_1',\cdots,P_n'$ or $S^2\times S^1$. So $-P'$ is not an oriented prime summand of $\widetilde{M}$, and $\widetilde{M}$ admits no orientation reversing self-homeomorphism.

\end{proof}

\end{document}